\documentclass[12pt]{article}
\usepackage[a4paper,left=3cm, right=3cm,top=3.5cm, bottom=4.25cm]{geometry}
\usepackage{amsmath,amssymb,amsfonts,amsthm}
\usepackage{epsfig}
\usepackage{tikz}

\let\oldbibliography\thebibliography
\renewcommand{\thebibliography}[1]{%
  \oldbibliography{#1}%
  \setlength{\itemsep}{2pt}%
}

\usepackage[]{hyperref}
\hypersetup{colorlinks=true, citecolor=black,%
                filecolor=black,%
               linkcolor=black,%
               urlcolor=blue}

\newtheorem{theorem}{Theorem}[section]

\newtheorem{proposition}[theorem]{Proposition}

\newtheorem{corollary}[theorem]{Corollary}

\newtheorem*{SPGT}{Strong Perfect Graph Theorem}

\usetikzlibrary[shapes]

\def \QD1 {\hfill $\spadesuit$}

\newcommand{\DF}[1]{{\bf #1\/}}

\newcommand{\set}[2]{\{#1 \;|\; #2 \}}
\newcommand{\ems}{\varnothing}
\newcommand{\sm}{\setminus}

\newcommand{\De}{\Delta}
\newcommand{\de}{\delta}
\newcommand{\pa}{\chi}
\newcommand{\om}{\omega}
\newcommand{\al}{\alpha}

\newcommand{\Ga}{\Gamma}
\newcommand{\p}{\ell}

\newcommand{\cSD}{{\cal S}{\cal D}}
\newcommand{\cMG}{{\cal MG}}
\newcommand{\cGD}{{\cal GD}}

\newcommand{\CO}{{\cal CO}}

\newcommand{\f}{\varphi}
\newcommand{\fin}{\varphi^{-1}}

\newcommand{\has}{\nabla}

\newcommand{\nat}{\mathbb{N}}
\newcommand{\nato}{\mathbb{N}_0}
\newcommand{\ganz}{\mathbb{Z}}

\newcommand{\PP}  {{\sf P}}
\newcommand{\NP}  {{\sf NP}}
\newcommand{\NPC} {{\NP}-complete}

\newcommand{\jou}[4]{{\em #1} {\bf #2} (#3) #4.}
\def \JCTB {J. Combin. Theory \, Ser.~B}
\def \DM {Discrete Math.}

\newcommand{\arxiv}[1]{\href{http://arxiv.org/abs/#1}{\texttt{arXiv:#1}}}

\numberwithin{equation}{section}

\begin{document}
\title{\bf Point partition numbers: \\perfect graphs}

\author{Justus von Postel \qquad Thomas Schweser \qquad  Michael Stiebitz\\[0.5ex]
\small Institute of Mathematics\\[-0.8ex]
\small Technische Universit\"at Ilmenau\\[-0.8ex]
\small D-98684 Ilmenau, Germany\\
\small\tt \{justus.von-postel,thomas.schweser,michael.stiebitz\}@tu-ilmenau.de}
\date{}
\maketitle
\begin{abstract}
Graphs considered in this paper are finite, undirected and without loops, but with multiple edges. For an integer $t\geq 1$, denote by $\mathcal{MG}_t$ the class of graphs whose maximum multiplicity is at most $t$. A graph $G$ is called strictly $t$-degenerate if every non-empty subgraph $H$ of $G$ contains a vertex $v$ whose degree in $H$ is at most $t-1$. The point partition number $\chi_t(G)$ of $G$ is smallest number of colors needed to color the vertices of $G$ so that each vertex receives a color and vertices with the same color induce a strictly $t$-degenerate subgraph of $G$. So $\chi_1$ is the chromatic number, and $\chi_2$ is known as the point aboricity. The point partition number $\chi_t$ with $t\geq 1$ was introduced by Lick and White. If $H$ is a simple graph, then $tH$ denotes the graph obtained from $H$ by replacing each edge of $H$ by $t$ parallel edges. Then $\omega_t(G)$ is the largest integer $n$ such that $G$ contains a $tK_n$ as a subgraph. Let $G$ be a graph belonging to $\cMG_t$. Then $\omega_t(G)\leq \chi_t(G)$ and we say that $G$ is $\chi_t$-perfect if every induced subgraph $H$ of $G$ satisfies $\omega_t(H)=\chi_t(H)$. Based on the Strong Perfect Graph Theorem due to Chudnowsky, Robertson, Seymour and Thomas, we give a characterization of $\chi_t$-perfect graphs of $\mathcal{MG}_t$ by a set of forbidden induced subgraphs. We also discuss some complexity problems for the class of $\chi_t$-critical graphs.
\end{abstract}



\section{Introduction}

In this paper we extend the theory of perfect graphs to graphs having multiple edges. For this purpose we replace the chromatic number $\chi$ by the point partition number (respectively $t$-chromatic number) $\pa_t$  introduced in the 1970s by Lick and White \cite{LickW1970}.

\section{Notation for graphs}

For integers $k$ and $\ell$, let $[k,\ell]=\set{x\in \ganz}{k \leq x \leq \ell}$, let $\nat$ be the the set of positive integers and $\nato=\nat \cup \{0\}$. By a \DF{graph} we mean a finite undirected graph with multiple edges, but without loops. For a graph $G$, let $V(G)$ and $E(G)$ denote the \DF{vertex set} and the \DF{edge set} of $G$, respectively. The number of vertices of $G$ is called the \DF{order} of $G$ and is denoted by $|G|$. A graph $G$ is called \DF{empty} if $|G|=0$, in this case we also write $G=\ems$. For a vertex $v\in V(G)$ let $E_G(v)$ denote the set of edges of $G$ \DF{incident} with $v$. Then $d_G(v)=|E_G(v)|$  is the \DF{degree} of $v$ in $G$. As usual, $\de(G)=\min_{v\in V(G)}d_G(v)$ is the \DF{minimum degree} and $\De(G)=\max_{v\in V(G)}d_G(v)$ is the \DF{maximum degree} of $G$. For different vertices $u, v$ of $G$, let $E_G(u,v)=E_G(u)\cap E_G(v)$ be the set of edges incident with $u$ and $v$. If $e\in E_G(u,v)$ then we also say that $e$ is an edge of $G$ \DF{joining} $u$ and $v$. Furthermore, $\mu_G(u,v)=|E_G(u,v)|$ is the \DF{multiplicity} of the vertex pair $u,v$; and $\mu(G)=\max_{u\not=v}\mu_G(u,v)$ is the \DF{maximum multiplicity} of $G$. The graph $G$ is said to be \DF{simple} if $\mu(G)\leq 1$. For $X,Y \subseteq V(G)$, denote by $E_G(X,Y)$ the set of all edges of $G$ joining a vertex of $X$ with a vertex of $Y$. If $G'$ is a \DF{subgraph} of $G$, we write $G'\subseteq G$. The \DF{subgraph} of $G$ \DF{induced by} the vertex set $X$ with $X\subseteq V(G)$ is denoted by $G[X]$, i.e., $V(G[X])=X$ and $E(G[X])=E_G(X,X)$. Furthermore, $G-X=G[V(G)\sm X]$. For a vertex $v$, let $G-v=G-\{v\}$. For $F\subseteq E(G)$, let $G-F$ denote the subgraph of $G$ with vertex set $V(G)$ and edge set $E(G)\sm F$. For an edge $e\in E(G)$, let $G-e=G-\{e\}$. We denote by $K_n$ the \DF{complete graph} of order $n$ with $n\geq 0$, and by $C_n$ the \DF{cycle} of order $n$ with $n\geq 2$. A cycle is called \DF{odd} or \DF{even} depending on whether its order is odd or even. The order of a cycle is also called its \DF{length}.

\section{Point partition numbers}

In what follows let $t\in \nat$. Denote by $\cMG_t$ the class of graphs $G$ with $\mu(G)\leq t$. So $\cMG_1$ is the class of simple graphs. If $G$ is a graph, then $H=tG$ denotes the graph obtained from $G$ by replacing each edge of $G$ by $t$ parallel edges, that is, $V(H)=V(G)$ and for any two different vertices $u,v\in V(G)$ we have $\mu_H(u,v)=t\mu_G(u,v)$. The graph $H=tG$ is called a \DF{$t$-uniform inflation} of $G$.

Let $G$ be an arbitrary graph. We call $G$ \DF{strictly $t$-degenerate} if every non-empty subgraph $H$ of $G$ has a vertex $v$ such that $d_{H}(v)\leq t-1$. Let $\cSD_t$ denote the class of strictly $t$-degenerate graphs. Clearly, $\cSD_1$ is the class of edgeless graphs, and $\cSD_2$ is the class of forests.

A \DF{coloring} of $G$ with \DF{color set} $\Ga$ is a mapping $\f:V(G) \to \Ga$ that assigns to each vertex $v\in V(G)$ a \DF{color} $\f(v)\in \Ga$. For a color $c\in \Ga$, the preimage $\fin(c)=\set{v\in V(G)}{\f(v)=c}$ is called a \DF{color class} of $G$ with respect to $\f$. A subgraph $H$ of $G$ is called \DF{monochromatic} with respect to $\f$ if $V(H)$ is a subset of a color class of $G$ with respect to $\f$.

A coloring $\f$ of $G$ with color set $\Ga$ is called an \DF{$\cSD_t$-coloring} of $G$ if for each color $c\in \Ga$ the subgraph of $G$ induced by the color class $\fin(c)$ belongs to $\cSD_t$. We denote by $\CO_t(G,k)$ the set of $\cSD_t$-colorings of $G$ with color set $\Ga=[1,k]$. The \DF{point partition number} $\pa_t(G)$ of the graph $G$ is defined as the least integer $k$ such that $\CO_t(G,k)\not=\ems$. Note that $\pa_1$ equals the \DF{chromatic number} $\chi$, and $\pa_2$ is known as the \DF{point aboricity}.

The graph classes $\cSD_t$ and the corresponding coloring parameters $\pa_t$ with $t\geq 1$ were introduced in 1970 by Lick and White \cite{LickW1970}. Bollob\'as and Manvel \cite{BollobasM79} used the term \DF{$t$-chromatic number} for the parameter $\pa_t$. The point aboricity $\pa_2$ was already introduced in 1968 by Hedetniemi \cite{Hedetniemi68}.

Clearly, an $\cSD_t$-coloring of a graph $G$ induces an $\cSD_t$-coloring with the same color set for each of its subgraphs, and so
\begin{equation}
\label{Equ:monotone}
H\subseteq G \mbox{ implies } \pa_t(H)\leq \pa_t(G).
\end{equation}
Furthermore, it is easy to check that if we delete a vertex or an edge from a graph, then its $t$-chromatic number decreases by at most one. Clearly, $sK_2\not\in \cSD_t$ if and only if $s\geq t$. Consequently, if a graph $G$ has a vertex pair $(u,v)$ such that $\mu_G(u,v)\geq t+1$ and $e\in E_G(u,v)$, then $\pa_t(G-e)=\pa_t(G)$. So it suffices to investigate the $t$-chromatic number for graphs belonging to $\cMG_t$. Furthermore, if $G$ is a simple graph, then $tG\in \cMG_t$ and a coloring $\f$ of $G$ with color set $\Ga=[1,k]$ satisfies
\begin{equation}
\label{Equ:co1=cot}
\f\in \CO_1(G,k) \mbox{ if and only if } \f\in \CO_t(tG,k),
\end{equation}
which implies that
\begin{equation}
\label{Equ:pa1=pat}
\pa(G)=\pa_t(tG).
\end{equation}

Let $G$ be a graph. A vertex set $X\subseteq V(G)$ is called an \DF{$\cSD_t$-set} of $G$ if $G[X]$ belongs to $\cSD_t$. In particular, an $\cSD_1$-set is also called an \DF{independent set} of $G$. Let $\al_t(G)$ be the maximum cardinality of an $\cSD_t$-set of $G$. Note that $\pa_t(G)$ is the least integer $k$ such that $V(G)$ has a partition into $k$ sets each of which is an  $\cSD_t$-set of $G$.
Consequently, $G$ satisfies
\begin{equation}
\label{Equ:pa_t+al_t}
|G|\leq \pa_t(G) \al_t(G).
\end{equation}
We call $X\subseteq V(G)$ a \DF{$t$-fold clique} of $G$ if $\mu_G(u,v)\geq t$ for any pair $(u,v)$ of distinct vertices of $X$. A $1$-fold clique of $G$ is also called a \DF{clique} of $G$. Let $\om_t(G)$ the maximum cardinality of a $t$-fold clique of $G$. If $G\in \cMG_t$, then a vertex set $X\subseteq V(G)$ of cardinality $n$ is a $t$-fold clique of $G$ if and only if $G[X]$ is a $tK_n$. Furthermore, \eqref{Equ:monotone} and \eqref{Equ:pa1=pat} implies that
\begin{equation}
\label{Equ:omt<=pat}
\om_t(G)\leq \pa_t(G).
\end{equation}

A graph $G\in \cMG_t$ is called \DF{$\pa_t$-perfect} if every induced subgraph $H$ of $G$ satisfies $\pa_t(H)=\om_t(H)$.

If a graph $G$ is $\pa_t$-perfect, then not only $G$, but all its induced subgraphs fulfill a min-max equation. This is one of the reasons why $\pa_t$-perfect graphs are interesting objects for graph theorists. The study of $\pa_1$-perfect graphs has attracted a lot of attention over the last six decades and was mainly motivated by a conjecture proposed in the 1960s by Berge \cite{Berge63}. However, it took more than forty years until Chudnovsky, Robertson, Seymour, and Thomas \cite{CRST06} succeeded in proving this conjecture; the result is now commonly known as the Strong Perfect Graph Theorem (SPGT). The SPGT provides a characterization of $\pa_1$-perfect graphs by forbidden induced subgraphs.

\begin{SPGT}
A simple graph is $\pa_1$-perfect if and only if it contains no odd cycle of length at least five, or its complement, as an induced subgraph.
\end{SPGT}

\section{Characterizing $\pa_t$-perfect graphs}

We need some more notation. Let $G$ be a graph belonging to $\cMG_t$. We call a graph $H$ the \DF{$t$-complement} of $G$, written $H=\overline{G}^t$, if $V(H)=V(G)$, $E(G)\cap E(H)=\ems$, and $\mu_H(u,v)+\mu_G(u,v)=t$ for every pair $(u,v)$ of distinct vertices of $G$. In particular, for $t=1$,
the $1$-complement is the ordinary \DF{complement} $\overline{G}$ of the simple graph $G$. Clearly, $H=\overline{G}^t$ if and only if $G=\overline{H}^t$, and $H$ is an induced subgraph of $G$ if and only if $\overline{H}^t$ is an induced subgraph of $\overline{G}^t$. If $G$ has order $n$, then $G\cup \overline{G}^t=tK_n$. Furthermore, let $S_t(G)$ denote the simple graph with vertex set $V(S_t(G))=V(G)$ and edge set $E(S_t(G))=\set{uv}{\mu_G(u,v)=t}$. Note that $G$ satisfies
\begin{equation}
\label{Equ:om_t(S)=om}
\om_t(G)=\om_1(S_t(G)).
\end{equation}

\begin{theorem}
\label{Theorem:pa2-perfect}
For any graph $G\in \cMG_2$ the following statements are equivalent:
\begin{itemize}
\item[{\rm (a)}] The graph $G$ is $\pa_2$-perfect.
\item[{\rm (b)}] The graph $S_2(G)$ is $\pa_1$-perfect and the graph $G$ contains no cycle of length at least three as an induced subgraph.
\item[{\rm (c)}] The graph $S_2(G)$ contains no odd cycle of length at least five, or its complement, as an induced subgraph, and the graph $G$ contains no cycle of length at least three as an induced subgraph.
\item[{\rm (d)}] $G$ contains no induced subgraph $H$ such that $S_2(H)$ is an odd cycle of length at least five, or its complement, or $H$ is a cycle of length at least three.
\end{itemize}
\end{theorem}
\begin{proof}
To show that (a) implies (b), suppose that $G$ is $\pa_2$-perfect. If $S_2(G)$ is not $\pa_1$-perfect, then there exists an induced subgraph $H$ of $S_2(G)$ such that $\om_1(H)<\pa_1(H)$. For $X=V(H)$, we have $S_2(G[X])=H$. Then it follows from \eqref{Equ:om_t(S)=om} that
$$\om_2(G[X])=\om_1(S_2(G[X]))=\om_1(H)<\pa_1(H)=\pa_1(S_2(G[X]))\leq \pa_2(G[X]),$$
which implies that $G$ is not $\pa_2$-perfect, a contradiction. If $G$ contains a cycle $C_n$ with $n\geq 3$ as an induced subgraph, then $\om_2(C_n)=1<2=\pa_2(C_n)$, and so $G$ is not $\pa_2$-perfect, a contradiction, too. This shows that (a) implies (b). To show the converse implication, suppose that $S_2(G)$ is $\pa_1$-perfect, but $G$ is not $\pa_2$-perfect. Our aim is to show that $G$ contains a cycle $C_n$ with $n\geq 3$ as an induced subgraph. Since $G$ is not $\pa_2$-perfect, there is an induced subgraph $H$ of $G$ such that $\om_2(H)<\pa_2(H)$. Let $k=\om_2(H)=\om_1(S_2(H))$ (see \eqref{Equ:om_t(S)=om}). Clearly, $S_2(H)$ is an induced subgraph of $S_2(G)$, and so $S_2(H)$ is $\pa_1$-perfect as $S_2(G)$ is $\pa_1$-perfect. Hence there is a coloring $\f\in \CO_1(S_2(H),k)$. Then \eqref{Equ:co1=cot} implies that $\f\in \CO_2(2S_2(H),k)$. As $k<\pa_2(H)$, $\f\not\in \CO_2(H,k)$ and so there is a color $c\in [1,k]$ such that $H[\fin(c)]\not\in \cSD_2$. As $\f\in \CO_1(S_2(H),k)$, $H[\fin(c)]$ contains a cycle, but no $C_2$. Consequently, $H[\fin(c)]$ contains an induced cycle of length at least three, which is also an induced cycle of $G$. This completes the proof that (b) implies (a). The equivalence of (b) and (c) follows from the SPGT, the equivalence of (c) and (d) is evident.
\end{proof}

Statement (d) of the above theorem provides a characterization of the class of $\pa_t$-perfect graphs by an infinite family of forbidden induced subgraphs. As an immediate consequence of the proof of Theorem~\ref{Theorem:pa2-perfect} we obtain the following two corollaries.

\begin{corollary}
\label{Corollary:S2-col=gcol}
Let $G\in \cMG_2$ be a $\pa_2$-perfect graph. Then every $\cSD_1$-coloring of $S_2(G)$ is a $\cSD_2$-coloring of $G$ with the same color set.
\end{corollary}

\begin{corollary}
\label{Corollary:Oreconstruction}
Let $G\in \cMG_2$ be graph. Then $G$ is $\pa_2$-perfect, or there are three distinct vertices $u,v$ and $w$ of $G$ such that $\mu_G(u,v)\leq 1, \mu_G(v,w)\leq 1$, and $\mu_G(u,w)\geq 1$.
\end{corollary}

Corollary~\ref{Corollary:S2-col=gcol} implies that there is a polynomial time algorithm that computes for a given $\pa_2$-perfect graph $G$ an optimal $\cSD_2$-coloring of $G$. Clearly, we can compute the simple graph $G'=S_2(G)$ in polynomial time and Theorem~\ref{Theorem:pa2-perfect} implies that $G'$ is $\pa_1$-perfect. Then it follows from a result by Gr\"otschel, Lov\'asz, and Schrijver \cite{GLS1981} that an optimal $\cSD_1$-coloring $\f$ of $G'$ can be computed in polynomial time. Then $\f$ is an optimal $\cSD_2$-coloring of $G$ (by Corollary~\ref{Corollary:S2-col=gcol} and \eqref{Equ:om_t(S)=om}).

Corollary~\ref{Corollary:Oreconstruction} is interesting as we wanted to use it to proof a Haj\'os-type result for the point aboricity $\pa_2$. In 1961 Haj\'os \cite{Hajos61} proved that, for any fixed integer $k\geq 3$, a simple graph $G$ has chromatic number at least $k$ if and only if $G$ contains a
\DF{$k$-constructible} subgraph, that is, a graph that can be obtained from disjoint copies of $K_k$ by repeated application of the Haj\'os join and the identification of non-adjacent vertices. Recall that the \DF{Haj\'os join} of two disjoint graphs $G_1$ and $G_2$ with edges $u_1v_1\in E(G_1)$ and $u_2v_2 \in E(G_2)$ is the graph $G$ obtained from the union $G_1 \cup G_2$ by deleting both edges $u_1v_1$ and $u_2v_2$, identifying $v_1$ with $v_2$, and adding the new edge $u_1u_2$; we then write $G=G_1\has G_2$. The "if" implication of Haj\'os' theorem follows from the facts that $\pa_1(K_k)=k$, $\pa_1(G_1\has G_1)\geq \max \{\pa_1(G_1),\pa_2(G_2)\}$ (provided that $E(G_i)\not=\ems$), and
$\pa_1(G/I)\geq \pa_1(G)$, where $G/I$ denotes the (simple) graph obtained from $G$ by identifying an independent set $I$ of $G$ to a single vertex. The proof of the "only if" implication is by reductio ad absurdum. So we consider a simple graph $G$ with $\pa_1(G)\geq k$ and without a $k$-constructible subgraph. The graph $G$ may be assumed to be maximal in the sense that the addition of any edge $e\in E(\overline{G})$ to $G$ gives rise to a $k$-constructible subgraph $G_e$ of $G+e$ with $e\in E(G_e)$. If non-adjacency is  an equivalence relation on $V(G)$, then the number of equivalence classes is at least $k$ (as $\pa_1(G)\geq k$), which implies that $G$ contains a $K_k$, a contradiction. Therefore, there are three vertices $u,v$ and $w$ such that $uv, vw\in E(\overline{G})$ and $uw\in E(G)$. Then there are two $k$-constructible graphs $G_{uv}$ and $G_{vw}$. Now let $G'=(G_{uv}-uv) \cup (G_{vw}-vw)+uw$. Then $G'$ is a subgraph of $G$ which can be obtained from disjoint copies of $G_{uv}$ and $G_{vw}$ by removing the copies of the edges $uv$ and $vw$, identifying the two copies of $v$ and adding the copy of the edge $uw$. Then, for each vertex $x$ belonging to both $G_{uv}$ and $G_{vw}$ we identify the two copies of $x$, thereby obtaining the $k$-constructible subgraph $G'$ of $G$, a contradiction.

That the Haj\'os join well behaves with respect to the point aboricity $\pa_2$ was proved by the authors in \cite{PostelSS2020}; the Haj\'os join not only preserves the point aboricity, but also criticality. So we were hopeful to establish a counterpart of Haj\'os' theorem for the point aboricity, with $2K_k$ as the basic graphs. For the proof of the "only if" implication we could use Corollary~\ref{Corollary:Oreconstruction}. However, we were not able to control the identification operation for graphs in $\cMG_2$ to handle the "if" implication. So we did not succeed in finding a constructive characterization for the class of graphs $G\in \cMG_2$ with $\pa_2(G)\geq k$.

\medskip

The proof of Theorem~\ref{Theorem:pa2-perfect} can easily be extended to obtain a characterization for the class of $\pa_t$-perfect graphs with $t\geq 3$ by a family of forbidden induced subgraphs. For $t\in \nat$ with $t\geq 2$, let $\cGD_t$ denote the class of connected graphs $G\in \cMG_{t-1}$ with $\de(G)\geq t$. Note that $\cGD_t \cap \cSD_t=\ems$.

\begin{theorem}
\label{Theorem:pat-perfect}
Let $t\in \nat$ with $t\geq 2$. For any graph $G\in \cMG_t$ the following statements are equivalent:
\begin{itemize}
\item[{\rm (a)}] The graph $G$ is $\pa_t$-perfect.
\item[{\rm (b)}] The graph $S_t(G)$ is $\pa_1$-perfect and no induced subgraph of $G$ belongs to
$\cGD_t$.
\item[{\rm (c)}] $G$ contains no induced subgraph $H$ such that $S_2(H)$ is an odd cycle of length at least five, or its complement, or $H\in \cGD_t$.
\end{itemize}
\end{theorem}
\begin{proof}
To show that (a) implies (b), suppose that $G$ is $\pa_t$-perfect. If $S_t(G)$ is not $\pa_1$-perfect, then there exists an induced subgraph $H$ of $S_t(G)$ such that $\om_1(H)<\pa_1(H)$. For $X=V(H)$, we have $S_t(G[X])=H$. Then it follows from \eqref{Equ:om_t(S)=om} that
$$\om_t(G[X])=\om_1(S_t(G[X]))=\om_1(H)<\pa_1(H)=\pa_1(S_t(G[X]))\leq \pa_t(G[X]),$$
which implies that $G$ is not $\pa_t$-perfect, a contradiction. If $G$ has an induced subgraph $H\in \cGD_t$, then $\om_t(H)=1$ (as $H\in \cMG_{t-1}$) and $\pa_t(H)\geq 2$ (as $H\not\in \cSD_t$), which implies that $G$ is not $\pa_2$-perfect, a contradiction, too. This shows that (a) implies (b). To show the converse implication, suppose that $S_t(G)$ is $\pa_1$-perfect, but $G$ is not $\pa_t$-perfect. Our aim is to show that $G$ contains a graph $H\in \cGD_t$ as an induced subgraph. Since $G$ is not $\pa_t$-perfect, there is an induced subgraph $G'$ of $G$ such that $\om_t(G')<\pa_t(G')$. Let $k=\om_t(G')=\om_1(S_t(G'))$ (see \eqref{Equ:om_t(S)=om}). Clearly, $S_t(G')$ is an induced subgraph of $S_t(G)$, and so $S_t(G')$ is $\pa_1$-perfect as $S_t(G)$ is $\pa_1$-perfect. Hence there is a coloring $\f\in \CO_1(S_t(G'),k)$. Then \eqref{Equ:co1=cot} implies that $\f\in \CO_t(tS_t(G'),k)$. As $k<\pa_t(G')$, $\f\not\in \CO_t(G',k)$ and so there is a color $c\in [1,k]$ such that $G'[\fin(c)]\not\in \cSD_t$. As $\f\in \CO_1(S_t(G'),k)$, $G'[\fin(c)]$ contains no $tK_2$ and so $G'[\fin(c)]\in \cMG_{t-1}$. Then we conclude that $G'[\fin(c)]$ contains an induced subgraph $H$ such that $H$ is connected and $\de(H)\geq t$. Then $H$ is an induced subgraph of $G$ belonging to $\cGD_t$, as required. This completes the proof that (b) implies (a). The equivalence of (b) and (c) follows from the SPGT.
\end{proof}

\begin{corollary}
\label{Corollary:St-col=gcol}
Let $G\in \cMG_t$ be a $\pa_t$-perfect graph with $t\geq 2$. Then every $\cSD_1$-coloring of $S_t(G)$ is a $\cSD_t$-coloring of $G$ with the same color set.
\end{corollary}

\begin{corollary}
\label{Corollary:t-optimal}
Let $t\geq 2$. Then there exists a polynomial time algorithm that computes for any $\pa_t$-perfect graph $G\in \cMG_t$ an optimal $\cSD_t$-coloring.
\end{corollary}

It is well known that many decision problems that are {\NPC} for the class $\cMG_1$ are polynomial solvable for the class of $\pa_1$-perfect graphs, e.g. the clique problem, the independent set problem, the coloring problem, and the clique covering problem (see \cite[Corollaries 9.3.32, 9.3.33, 9.4.8, Theorem 9.4.3]{GLS1993}).

Let $G\in \cMG_t$ be a $\pa_t$-perfect graph with $t\geq 2$. Then $G'=S_t(G)$ belongs to $\cMG_1$, and $G'$ is a $\pa_1$-perfect graph and no induced subgraph of $G$ belongs to $\cGD_t$ (by Theorem~\ref{Theorem:pat-perfect}(b)). Now let $I\subseteq V(G)$. We claim that $G[I]\in \cSD_t$ if and only if $G'[I]\in \cSD_1$ (i.e., $I$ is an independent set of $G'$). If $G[I]\in \cSD_t$, then $G[I]$ does not contain a $tK_2$ as a subgraph, and so $I$ is an independent set of $G'=S_t(G)$. Now assume that $I$ is an independent set of $G'$. Then $G[I]\in \cMG_{t-1}$ and, since no induced subgraph of $G$ belongs to $\cGD_t$, it follows that $G[I]\in \cSD_t$. By a result due to Gr\"otschel, Lov\'asz, and Schrijver \cite{GLS1981} it is possible to find an independent set $I$ of $G$ with $|I|=\al_1(G')$ in polynomial time. Hence, we have the following result.

\begin{corollary}
\label{Corollary:alpha_t}
Let $t\geq 2$. Then there exists a polynomial time algorithm that computes for any $\pa_t$-perfect graph $G\in \cMG_t$ an induced subgraph $H$ of $G$ such that $H\in \cSD_t$ and $|H|=\al_t(G)$.
\end{corollary}

\section{A weak perfect graph theorem}

In 1972 Lov\'asz \cite{Lovasz72, Lovasz72b} proved the following result, which was proposed by A. Hajnal.

\begin{theorem} [Lov\'asz 1972]
\label{Theorem:pa_1-Lovasz}
A simple graph $G$ is $\pa_1$-perfect if and only if $|H|\leq \om_1(H)\al_1(H)$ for every induced subgraph $H$ of $G$.
\end{theorem}

On the one hand, this result is an immediate consequence of the SPGT. On the other hand, Lov\'asz gave a proof avoiding the use of the SPGT, in fact he proved it before the SPGT was established. In 1996 Gasparian \cite{Gasparian96} applied an argument from linear algebra in order to give a very short proof of Lova\'sz' result. We shall use Theorem~\ref{Theorem:pat-perfect} to extend Lov\'asz' theorem to $\pa_t$-perfect graphs.

\begin{theorem}
\label{Theorem:pa_t-Lovasz}
Let $t\geq 2$. A graph $G\in \cMG_t$ is $\pa_t$-perfect if and only if $|H|\leq \om_t(H)\al_t(H)$ for every induced subgraph $H$ of $G$.
\end{theorem}
\begin{proof}
First assume that $G$ is $\pa_t$-perfect. If $H$ is an induced subgraph of $G$, then $H$ is $\pa_t$-perfect, too. Based on \eqref{Equ:pa_t+al_t}, we obtain that $|H|\leq \pa_t(H)\al_t(H)= \om_t(H)\al_t(H)$. Thus the forward implication is proved.

To prove the backward implication assume that $G$ is not $\pa_t$-perfect. It then suffices to show that $G$ has an induced subgraph $H$ such that $|H| >\om_t(H)\al_t(H)$. By Theorem~\ref{Theorem:pat-perfect}, $G$ has an induced subgraph $H$ such that $S_t(H)$ is an odd cycle of length at least five, or its complement, or $H\in \cGD_t$. If $S_t(H)$ is an odd cycle of length $\ell\geq 5$, then $\om_t(H)=2$ and $\al_t(H)\leq (\ell-1)/2$, which leads to $|H|=\ell >\om_t(H)\al_t(H)$. If $S_t(H)$ is the complement of an odd cycle of length $\ell\geq 5$, then $\al_t(H)=2$ and $\om_t(H)\leq (\ell-1)/2$, which leads to $|H|=\ell >\om_t(H)\al_t(H)$. If $H\in \cGD_t$, then $\om_t(H)=1$ and $\al_t(H)\leq |H|-1$, which leads to $|H| >\om_t(H)\al_t(H)$.
\end{proof}

Note that if $G$ is a simple graph, then $H$ is an induced subgraph of $G$ if and only if $\overline{H}$ is an induced subgraph of $\overline{G}$. Furthermore, if $H$ is an induced subgraph of $G$, then $\al_1(\overline{H})=\om_1(H)$ and $\om_1(\overline{H})=\al_1(H)$. Consequently, Theorem~\ref{Theorem:pa_1-Lovasz} implies that a simple graph is $\pa_1$-perfect if and only if its complement is $\pa_1$-perfect. This immediate corollary of Lov\'asz' theorem, nowadays known as the Weak Perfect Graph Theorem (WPGT), was also conjectured by Berge in the 1960s. The WPGT has no direct counterpart for $\pa_t$-perfect graphs. However, we shall proof a result  for $\pa_2$-perfect graphs that might be considered as a weak $\pa_2$-perfect graph theorem (see Theorem~\ref{Theorem:wpgt}).

First we need some notation. Let $G$ be an arbitrary graph. We denote by $S(G)$ the \DF{underlying simple graph} of $G$, that is, $V(S(G))=V(G)$ and $E(S(G))=\set{uv}{\mu_G(u,v)\geq 1}$. Note that a vertex set $X\subseteq V(G)$ is a clique of $G$
if and only if $S(G)[X]$ is a complete graph. A subgraph $C$ of $G$ is called a \DF{simple cycle} of $G$ if $C$ is a cycle and any pair $(u,v)$ of vertices that are adjacent in $C$ satisfies $\mu_G(u,v)=1$. We say that $G$ is a \DF{normal graph} if $G$ contains no simple cycle whose vertex set is a clique of $G$. If $G$ is a normal graph and $H=S(G)$ is its underlying simple graph, then we also say that $G$ is a \DF{clique-acyclic inflation} of $H$.

\begin{proposition}
\label{Proposition:acyclic}
Let $G\in \cMG_2$ be a graph. Then $G$ contains no induced cycle if and only if $\overline{G}^2$ is a normal graph.
\end{proposition}
\begin{proof}
First assume that $G$ contains an induced cycle $C$. Then $\overline{C}^2$ is a simple cycle of $\overline{G}^2$ whose vertex set is a clique of $\overline{G}^2$, and so $\overline{G}^2$ is not normal. Now assume that $\overline{G}^2$ is not normal. Then $\overline{G}^2$contains a simple cycle whose vertex set is a clique of $\overline{G}^2$. Let $C$ be such a cycle whose length is minimum. Then $\overline{C}^2$ is an induced cycle of $G$.
\end{proof}

\begin{theorem}
\label{Theorem:wpgt}
Let $G\in \cMG_2$ be a graph. Then $G$ is a $\pa_2$-perfect graph if and only if $\overline{G}^2$ is a clique-acyclic inflation of a perfect graph.
\end{theorem}
\begin{proof}
Since $G\in \cMG_2$, it follows that $S(\overline{G}^2)=\overline{S_2(G)}$. By combining
Theorem~\ref{Theorem:pa2-perfect} with Proposition~\ref{Proposition:acyclic}, $G$ is a $\pa_2$-perfect graph if and only if $S_2(G)$ is a $\pa_1$-perfect graph and $\overline{G}^2$ is a normal graph. By the WPGT, this implies that $G$ is a $\pa_2$-perfect graph if and only if $\overline{S_2(G)}$ is a $\pa_1$-perfect graph and $\overline{G}^2$ is a normal graph, which is equivalent to $\overline{G}^2$ is a clique-acyclic inflation of a perfect graph.
\end{proof}

\section{Recognizing $\pa_2$-perfect graphs}

In 2005, Chudnovsky, Cornu\'ejols, Liu, Seymour, Vu\v{s}kovi\'c \cite{CCLSV05} proved that the decision problem whether a simple graph does not contain an odd cycle of length at least five, or its complement, as an induced subgraph belongs to the complexity class {\PP}. So the SPGT then implies that the recognition problem for perfect graphs is polynomial time solvable. To test whether a graph $G\in \cMG_2$ is $\pa_2$-perfect, by Theorem~\ref{Theorem:pa2-perfect} we have to test
\begin{enumerate}
  \item whether $S_2(G)$ is $\pa_1$-perfect, and
  \item whether $G$ does not contain an induced cycle of length at least three.
\end{enumerate}

While the first can indeed be tested efficiently (by the SPGT and \cite{CCLSV05}), the second is a co-{\NPC} problem (see Theorem~\ref{Theorem:NP-complete}). In 2012 Bang-Jensen, Havet, and Trotignon \cite[Theorem 11]{BangJensenHT2012} proved that the decision problem whether a digraph contains an induced directed cycle of length at least three is {\NPC}. Here a digraph may have antiparallel arcs, but no parallel arcs. We can use the same reduction as in \cite[Theorem 11]{BangJensenHT2012} (just by ignoring the directions) to prove the following result.

\begin{figure}[htbp]
\centering
\begin{tikzpicture}[vertex/.style={circle,minimum size=2mm,very thick, draw=black, fill=black, inner sep=0mm}]
\node[draw=none,minimum size=3cm,regular polygon,regular polygon sides=4, rotate=45] (a) {};

\node[vertex, label={above:$\overline{x}_i$}] (a1) at (a.corner 1) {};
\node[vertex, label={left:$a_i$}] (a2) at (a.corner 2){};
\node[vertex, label={below:$x_i$}] (a3) at (a.corner 3){};
\node[vertex, label={right:$b_i$}] (a4) at (a.corner 4){};

\path[-]
(a1) edge (a2)
(a2) edge (a3)
(a3) edge (a4)
(a4) edge (a1)
(a1) edge [line width=.8mm] (a3);

\begin{scope}[xshift=7cm]
\node[vertex, label={160:$\ell_{j2}$}](a0) {};
\node[draw=none,minimum size=3cm,regular polygon,regular polygon sides=4, rotate=45] (a) {};

\node[vertex, label={above:$\ell_{j1}$}] (a1) at (a.corner 1) {};
\node[vertex, label={left:$c_j$}] (a2) at (a.corner 2){};
\node[vertex, label={below:$\ell_{j3}$}] (a3) at (a.corner 3){};
\node[vertex, label={right:$d_j$}] (a4) at (a.corner 4){};

\path[-]
(a1) edge (a2)
(a2) edge (a3)
(a3) edge (a4)
(a2) edge (a0)
(a0) edge (a4)
(a4) edge (a1)
(a1) edge [line width=.8mm] (a0)
(a0) edge [line width=.8mm] (a3)
(a1) edge [line width=.8mm, bend left] (a3);

\end{scope}

\end{tikzpicture}
\caption{Variable gadget $VG(i)$ and claus gadget $CG(j)$. Bold edges represent parallel edges.}
\label{Figure:Gadgets}   
\end{figure}
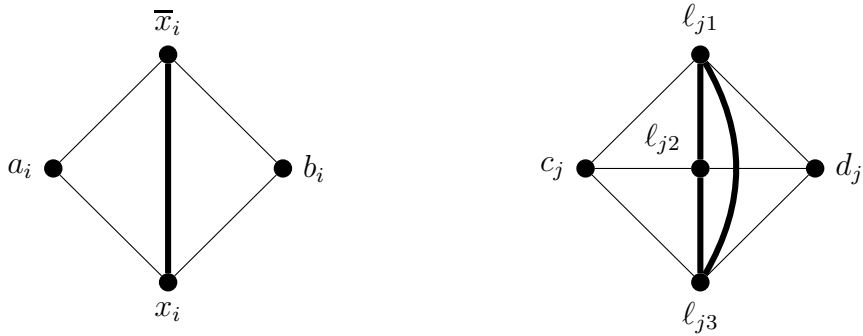

\begin{theorem}
\label{Theorem:NP-complete}
The decision problem whether a graph of $\cMG_2$ contains an induced cycle of length at least three is {\NPC}.
\end{theorem}
\begin{proof}
The reduction is from 3-SAT. Let $I$ be an instance of 3-SAT with variable $x_1, x_2, \ldots, x_n$ and clauses $C_1, C_2, \ldots, C_m$, where $n,m\in \nat$ and
$$C_j=(\p_{j1}\vee \p_{j2} \vee \p_{j3}) \mbox{ with } \p_{jk}\in \{x_1, x_2, \ldots, x_n, \bar{x}_1, \bar{x}_2, \ldots \bar{x}_n\}.$$
For each variable $x_i$ let $VG(i)$ be the variable gadget, and for each clause $C_j$ let $CG(j)$ be the clause gadget, as shown in Figure~\ref{Figure:Gadgets}. From the union of all the gadgets we form a graph $G(I)$ by adding edges according to the following two rules:
\begin{enumerate}
  \item We add the single edges $b_ia_{i+1}$ (with $i\in [1,n-1]$), $b_nc_1$, $d_jc_{j+1}$ (with $j\in [1,m-1]$), and $d_ma_1$.
  \item For each literal $\p_{jk}$ (which is either $x_i$ or $\bar{x}_i$) we add two parallel edges joining the vertex $\p_{jk}$ of the clause gadget $CG(j)$ with the vertex $\overline{\p_{jk}}$ in the variable gadget $VG(i)$.
\end{enumerate}
Similar as in the proof of \cite[Theorem 11]{BangJensenHT2012} it is easy to show that $G(I)$ has an induced cycle of length at least three if and only if $I$ is satisfiable.
\end{proof}

\begin{corollary}
\label{Corollary:NP-complete}
The decision problem whether a graph from $\cMG_2$ is $\pa_2$-perfect is co-{\NPC}.
\end{corollary}
\begin{proof}
We use the same reduction as in the proof of Theorem~\ref{Theorem:NP-complete}. So for an instance $I$ of 3-SAT we construct the graph $G(I)$. If necessary, we subdivide the edge $d_ma_1$, so the the graph $S_2(G(I))$ is bipartite and hence $\pa_1$-perfect. Then Theorem~\ref{Theorem:pa2-perfect} implies that $G(I)$ is $\pa_2$-perfect if and only if $G(I)$ has no induced cycle of length at least three.
\end{proof}

\section{Concluding remarks}

The results about $\pa_2$-perfect graphs, in particular, Theorems~\ref{Theorem:pa2-perfect} and \ref{Theorem:wpgt}, resemble the results about perfect digraphs (with respect to the directed chromatic number) obtained in 2015 by Andres and Hochst\"att\-ler \cite{AndresH2014}. The characterization of perfect digraphs was used by Bang-Jensen, Bellito, Schweser, and Stiebitz \cite{BangJensenBSS2019} to establish a Haj\'os-type result for the dichromatic number of digraphs. Andres \cite{Andres2012} started to investigate game-perfect graphs, based on maker-breaker games. This concept can also be extended to the point partition number.

\end{document}